\documentclass[11pt,epsf]{article}
\usepackage[dvips]{color}
\usepackage{indentfirst}
\newcommand{\bed}{\begin{displaymath}}
\newcommand{\eed}{\end{displaymath}}
\newcommand{\bea}{\bed\begin{array}{rl}}
\newcommand{\eea}{\end{array}\eed}
\newcommand{\barray}{\begin{array}{ll}}
\newcommand{\earray}{\end{array}}

\usepackage{color,latexsym,amsfonts,amssymb,amsmath,amsthm}
\usepackage{colordvi,amsfonts,amsmath,epsfig,subfigure,cite}
\usepackage{graphicx}
\usepackage{caption}
\newtheorem{theorem}{Theorem}[section]
\newtheorem{lemma}[theorem]{Lemma}
\newtheorem{remark}{Remark}[section]

\newtheorem{assumption}{Assumption}[section]

\begin{document}

\title{Convergence  of the tamed Euler scheme for stochastic differential equations with Piecewise Continuous Arguments under non-Lipschitz continuous coefficients}

\author{M.H. Song,\thanks{Corresponding author. Email: songmh@lsec.cc.ac.cn,~yulanlu@.hit.edu.cn,~mzliu@hit.edu.cn
\vspace{6pt}}~~Y.L. Lu,~~M.Z. Liu\\\vspace{6pt}{\em{Department of Mathematics,
Harbin Institute of~ Technology, Harbin, China, 150001}}\thanks{This work is supported by the NSF of P.R. China (No.11071050)}}
 \maketitle
\begin{abstract}
Recently, Martin Hutzenthaler pointed out that the explicit Euler method fails to converge strongly to the exact solution of a stochastic differential equation (SDE) with superlinearly growing and globally one sided Lipschitz drift coefficient. Afterwards, he proposed an explicit and easily implementable Euler method, i.e tamed Euler method, for such an SDE and showed that this method converges strongly with order of one half. In this paper, we use the tamed Euler method to solve the stochastic differential equations with piecewise continuous arguments (SEPCAs) with superlinearly growing coefficients and prove that this method is convergent with strong order one half.

 \end{abstract}
\section{Introduction}

\setcounter{equation}{0}
\noindent

Differential equations with piecewise continuous arguments(EPCAs) represent a hybrid of continuous and discrete dynamical systems and combine the properties of both differential and difference equations. They provide a mathematical modeling for physical and biological systems in which the rate of  the change of these systems depends on its past state. This kind of equations plays an important role in many branches of science and industry such as physics, biology and control theory, and has been initialed in \cite{j,w}. The general theory and basic results for EPCAs have by now been thoroughly investigated in the book of Wiener \cite{jo}.

Since many EPCAs can't be solved explicitly, computing numerical solutions and analysing their properties are necessary. The first work devoted to numerical study for EPCAs is the paper of Liu et al \cite{liu}. Subsequently, Song et al \cite{song} and Yang et al \cite{yang} studies the stability of $\theta$-methods for advanced equations and Runge-Kutta method for retarded equations, respectively. Liu and Gao \cite{gao,mzliu} gave the conditions under which the Runge-Kutta methods preserve the oscillation of linear EPCAs, and afterwards. Some authors considered the stability and oscillation of the numerical solutions \cite{msong,qwang,wangq}. Song and Liu \cite{mhsong} constructed the convergence improved linear multistep method for EPCAs.

However, systems are often influenced by environmental or some occasional events, which leads that the deterministic differential equations can't demonstrate the real world. In order to avoid this problem, many researchers turn to study SDEs. And there have been lots of results on both analytical and numerical solutions. The explicit Euler method (see \cite{Mao,Milstein,Kloden}) is most commonly used for approximating SDEs with global Lipschitz continuous coefficients. Unfortunately, the coefficients of large number of SDEs don't satisfy global Lipschitz condition \cite{Mil,Higham,zhang,jiang,xrmao,wmao,XRmao,lswang,Tak}. Higham et al \cite{Higham} showed that the explicit Euler method is convergent strongly when the coefficients of SDEs are local Lipschitz continuous and the $p$ moment of both exact and numerical solutions are bounded. But, Martin Hutzenthaler in \cite{mh1} proved that the explicit Euler method does not converge in the strong mean square sense to the exact solution of SDE with superlinearly growing and globally one-side Lipschitz continuous drift coefficient. To overcome this difficulty, he in \cite{mh} proposed a modified explicit Euler method, i.e. tamed Euler method in which the drift term is modified such that it is uniformly bounded, which is convergent strongly for such SDE.

Up to now, only a few people considered SEPCAs. Zhang and Song in \cite{Ling} investigated the strong convergence of explicit Euler method for SEPCAs when the coefficients are globally Lipschitz continuous or grow at most linearly. Moreover, Song and Zhang in \cite{zhangling} proved the convergence in probability of explicit Euler method under the local Lipschitz and Khasminskii-type conditions.

Throughout the whole paper, we investigate the numerical solution of the tamed Euler method to SEPCAs of the form
\begin{equation}\label{SEPCA}
 dx(t)=\mu(x(t),x([t]))dt+\sigma(x(t),x([t]))dB(t),~t\in[0,T],
 \end{equation}
with initial value $x(0)=\xi$, where $B(t)$ is a $r$-dimensional Brownian motion, $x(t)\in \mathbb{R}^d, \mu: \mathbb{R}^d\times \mathbb{R}^d\rightarrow \mathbb{R}^d $ is a continuously differentiable function, $\sigma: \mathbb{R}^d\times \mathbb{R}^d\rightarrow \mathbb{R}^{d\times r}$ satisfies global Lipschitz condition. $[\cdot]$ denotes the greatest-integer function.
\section{Notations and main theorem}
\setcounter{equation}{0}
\noindent
Let $(\Omega,\mathcal{F},\mathbb{P})$ be a completed probability space with a filtration $\{\mathcal{F}_t\}_{t\geq 0}$ satisfying the usual conditions and $B(t)$ be a $r$-dimensional Brownian motion defined on this probability space. Furthermore, for any real valued $\{\mathcal{F}_t\}$-adapted process $x(t)$, we use $\|x(t)\|_{L_p(\Omega,\mathbb{R})}$ to denote $\big(\mathbb{E}\|x(t)\|^p\big)^\frac{1}{p}$.

We use notations $\|x\|:=(x_1^2+x_2^2+\ldots+x_d^2)^\frac{1}{2}$, $\big<x,y\big>:=x_1\cdot y_1+x_2\cdot y_2+\ldots+x_d\cdot y_d$ for all $x=(x_1,x_2,\ldots,x_d)\in\mathbb{R}^d$, $y=(y_1,y_2,\ldots,y_d)\in\mathbb{R}^d$, and $\|A\|:=\sup_{x\in\mathbb{R}^d,\|x\|\leq 1}\|Ax\|$ for all $A\in\mathbb{R}^{r\times d}$. What's more, if $A$ is a matrix or vector, its transpose is defined by $A^T$. Set $\sum_{i=u}^n=0$ if $u>n$. In the whole paper, we make the following assumptions on the SEPCAs (\ref{SEPCA}).
\begin{assumption}\label{assumption}
Assume there exist non-negative constants $K, c$ such that for any $x, y, x_1 ,y_1,x_2$, $y_2\in \mathbb{R}^d$, the coefficients $\mu$ and $\sigma$ satisfy
\begin{eqnarray}
 \label{con1}\|\sigma(x_1,y_1)-\sigma(x_2,y_2)\|\leq K(\|x_1-x_2\|+\|y_1-y_2\|),\\
  \label{con2}\big<x_1-x_2, \mu(x_1,y)-\mu(x_2,y)\big>\leq K\|x_1-x_2\|^2,\\
 \label{con3}\|\mu(x,y_1)-\mu(x,y_2)\|\leq K\|y_1-y_2\|,\\
  \label{con4}\|\mu_x(x,y)\|\leq K(1+\|x\|^c).
\end{eqnarray}
\end{assumption}
\begin{remark}\label{remark}The conditions \eqref{con1}, \eqref{con2}, \eqref{con3} imply that there exists unique solution to equation \eqref{SEPCA}, and the $p$ moments of this solution is bounded in any finite interval $[0,T]$ (see \cite{Ling}).
\end{remark}
In the rest of this paper, let $h=\frac{1}{m}$ be given stepsize with integer $m\geq 1$. Grid points $t_n$ are defined as $t_n=nh,~n=0,1,\cdots.$ For simplicity, we assume $T=Nh$, then $N=Tm$. We consider the explicit tamed method for \eqref{SEPCA}, which is defined by taking $y_0=x(0)$ and, generally
\begin{equation}\label{yn}
y_{n+1}=y_n+\frac{\mu(y_{n},y_{[nh]m})h}{1+h\|\mu(y_{n},y_{[nh]m})\|}+\sigma(y_{n},y_{[nh]m})\Delta B_{n}
\end{equation}
where $\Delta B_{n}=B(t_{n+1})-B(t_{n})$, $y_n$ is the approximation to $x(t_n)$.

In order to formulate the convergence theorem for the tamed Euler method \eqref{yn}, we now introduce appropriate time continuous interpolations of the time discrete numerical approximations \eqref{yn}. More formally, let $y(t):[0,T]\times\Omega\rightarrow\mathbb{R}^d$, be a sequence of stochastic process given by
\begin{equation}\label{yt}
  y(t)=y_{n}+\int_{t_{n}}^t\frac{\mu(y_{n},y_{[nh]m})}{1+h\|\mu(y_{n},y_{[nh]m})\|}ds+\int_{t_{n}}^{t}\sigma(y_{n},y_{[nh]m})dB(s)
\end{equation}
for all $t\in[t_{n},t_{n+1}),~n\in\{0,1,\ldots,Tm-1\}$ and $m\in\mathbb{N}=\{1,2,...\}$. It is easy to get that $y(t)$ is $\{\mathcal{F}_t\}-adapted$ stochastic process. Now we can establish the main theorem of this paper.
\begin{theorem}\label{th2}
Suppose Assumption \ref{assumption} hold. Then for each $p\geq1$ there exists non-negative constant $C$ dependent on $p, K, r$ and $c$, but independent of $h$, such that
\begin{equation}\label{conclusion}
  \Big(\mathbb{E}\Big[\sup_{t\in[0,T]}\|x(t)-y(t)\|^p\Big]\Big)^{\frac{1}{p}}\leq Ch^{\frac{1}{2}}.
\end{equation}
Here $x(t)$ denotes the exact solution of \eqref{SEPCA}, $y(t)$ is the continuous approximation solution of tamed-Euler method.
\end{theorem}
The detailed proof of this theorem is given in Section 3.
\begin{remark}
Theorem \ref{th2} illustrates that the time continuous tamed-Euler approximations \eqref{yt} converge in the strong $L_p$-sense with the supremum over the time interval $[0,T]$ inside the expectation to the exact solution of \eqref{SEPCA} with the standard convergence order $0.5$.
\end{remark}
\section{Estimation of $p$-moments and proof of Theorem \ref{th2}}
\setcounter{equation}{0}
\noindent
First of all, we introduce several notations here
\begin{equation}\label{lambda}
  \lambda=(1+T+2\|\mu(0,0)\|+4K+2\|\sigma(0,0)\|)^4,
\end{equation}
\begin{equation}\label{alphan}
  \begin{split}
  \alpha_n=1_{\{\|y_{[nh]m}\|\leq 1,\|y_{n}\|\leq 1\}^c}\Big<\frac{y_{n}}{\sqrt{\|y_{n}\|^2+\|y_{[nh]m}\|^2}},\frac{\sigma(y_{n},y_{[nh]m})}{\sqrt{\|y_{n}\|^2+\|y_{[nh]m}\|^2}}\Delta B_{n}\Big>,
  \end{split}
\end{equation}
\begin{equation}\label{dn}
D_n=(\lambda+\|\xi\|+1)\exp\big(4T\sqrt{\lambda}+2\sup_{u\in\{-1,0,1,...,n-1\}} \sum_{i=u+1}^{n-1}(\sqrt{\lambda}\|\Delta B_i\|^2+\alpha_i)\big),
\end{equation}
and
\begin{equation}\label{wn}
 \Omega_n=\Big\{\omega\in \Omega\Big|\sup_{j\in\{0,1,\cdots,n-1\}}D_{j}(\omega)\leq m^{\frac{1}{2c}},\sup_{j\in\{0,1,\cdots,n-1\}}\|\Delta B_{j}\|\leq1\Big\},
\end{equation}
here $\Omega_n\in\mathcal{F}$ for all $n\in\{1,2,...,Tm\}$ and $m\in\mathbb{N}$.
In the following, we give some lemmas which will be useful to prove the main theorem \ref{th2}.
\begin{lemma}\label{lemma1}
If Assumption \ref{assumption} holds, and $y_{n}, D_{n}$ and $\Omega_{n}$ are defined by \eqref{yn}, \eqref{dn}, \eqref{wn}. Then
\begin{equation}\label{dom}
  1_{\Omega_n}\|y_n\|\leq D_n
\end{equation}
for all $n\in\{0,1,\ldots,Tm\}$ and $m\in\mathbb{N}$.
\end{lemma}
\begin{proof}
According to the definition of $\Omega_n$, we can see $\|\Delta B_{n}\|\leq 1$ on $\Omega_{n+1}$ for all $n\in\{0,1,\ldots,Tm-1\}$ and $ m\in\mathbb{N}$. So the assumption \eqref{assumption} implies that
\begin{equation}\label{fore1}
\begin{split}
\|y_{n+1}\|\leq& \|y_{n}\|+h\|\mu(y_{n},y_{[nh]m})\|+\|\sigma(y_{n},y_{[nh]m})\Delta B_{n}\|\\
\leq&\|y_{n}\|+hK(1+\|y_{n}\|^c)\|y_{n}\|+hK\|y_{[nh]m}\|+h\|\mu(0,0)\|
\\&+K(\|y_{n}\|+\|y_{[nh]m}\|)\|\Delta B_{n}\|+\|\sigma(0,0)\| \|\Delta B_{n}\|
\\
\leq&1+3TK+T\|\mu(0,0)\|+2K+\|\sigma(0,0)\|\leq \lambda
\end{split}
\end{equation}
on $\Omega_{n+1}\cap\{\|y_{n}\|\leq 1, \|y_{[nh]m}\|\leq 1\}$ for all $n\in \{0,1,\ldots,Tm-1\}$ and $m\in \mathbb{N}$.
Moreover, the estimate $2ab\leq a^2+b^2$ for all $a,b\in \mathbb{R}$ show that
\begin{equation}\label{fore2}
  \begin{split}
  \|y_{n+1}\|^2=&\|y_{n}\|^2+\frac{\|\mu(y_{n},y_{[nh]m})\|^2h^2}{(1+h\|\mu(y_{n},y_{[nh]m})\|)^2}+\|\sigma(y_{n},y_{[nh]m})\Delta B_{n}\|^2
  \\
  &+2\Big<\frac{\mu(y_{n},y_{[nh]m})h}{1+h\|\mu(y_{n},y_{[nh]m})\|},\sigma(y_{n},y_{[nh]m})\Delta B_{n}\Big>
  \\
  &+2\Big<y_{n},\frac{\mu(y_{n},y_{[nh]m})h}{1+h\|\mu(y_{n},y_{[nh]m})\|}+\sigma(y_{n},y_{[nh]m})\Delta B_{n}\Big>
  \\
  \leq&\|y_{n}\|^2+2\|\mu(y_{n},y_{[nh]m})\|^2h^2+2\|\sigma(y_{n},y_{[nh]m})\|^2\|\Delta B_{n}\|^2\\
&+2\frac{\big<y_{n},\mu(y_{n},y_{[nh]m})\big>h}{1+h\|\mu(y_{n},y_{[nh]m})\|}+2\big<y_{n},\sigma(y_{n},y_{[nh]m})\Delta B_{n}\big>
\end{split}
\end{equation}
on $\Omega$ for all $n\in\{0,1,\ldots,Tm-1\}$ and all $m\in\mathbb{N}$. The condition \eqref{con3}, \eqref{con4} implies that

\begin{equation}\label{mu1}
  \begin{split}
  \|\mu(x,y)\|^2\leq&3\|\mu(x,y)-\mu(0,y)\|^2+3\|\mu(0,y)-\mu(0,0)\|^2+3\|\mu(0,0)\|^2\\
  \leq&3K^2(1+\|x\|^c)^2\|x\|^2+3K^2\|y\|^2+3\|\mu(0,0)\|^2\\
  \leq&3K^2(1+m^{\frac{1}{2}})^2\|x\|^2+3K^2\|y\|^2+3\|\mu(0,0)\|^2\\
 \leq& m(15K^2+3\|\mu(0,0)\|^2)(\|x\|^2+\|y\|^2)\\
  \leq&m\sqrt{\lambda}(\|x\|^2+\|y\|^2)
  \end{split}
\end{equation}
for all $x,y\in\mathbb{R}^d$ with $\{\|x\|\leq 1, \|y\|\leq 1\}^c\cap\{\|x\|\leq m^{\frac{1}{2c}},\|y\|\leq m^{\frac{1}{2c}}\}$. Using the condition \eqref{con1}, the following estimate is obtained.
\begin{equation}\label{sigma}
\begin{split}
\|\sigma(x,y)\|^2&\leq(\|\sigma(x,y)-\sigma(0,0)\|+\|\sigma(0,0)\|)^2
\\
&\leq(K\|x\|+K\|y\|+\|\sigma(0,0)\|)^2
\\
&\leq2(K+\|\sigma(0,0)\|)^2(\|x\|^2+\|y\|^2)
\\
&\leq\sqrt{\lambda}(\|x\|^2+\|y\|^2)
\end{split}
\end{equation}
for all $x,y\in\mathbb{R}^d$ with $\{\|x\|\leq 1, \|y\|\leq 1\}^c$. Moreover, condition \eqref{con2} yields that
\begin{equation}\label{xmu}
\begin{split}
 \big <x,\mu(x,y)\big>\leq&K\|x\|^2+\|x\|\|\mu(0,y)\|
  \\
  \leq&K\|x\|^2+\frac{1}{2}\|x\|^2+K^2\|y\|^2+\|\mu(0,0)\|^2
  \\
  \leq&(K^2+K+1+\|\mu(0,0)\|^2)(\|x\|^2+\|y\|^2)
  \\
  \leq&\sqrt{\lambda}(\|x\|^2+\|y\|^2)
  \end{split}
\end{equation}
for all $x,y\in\mathbb{R}^d$ with $\{\|x\|\leq 1, \|y\|\leq 1\}^c$. Combining \eqref{mu1}, \eqref{sigma} and \eqref{xmu}, \eqref{fore2} yields that
\begin{equation}\label{fore3}
\begin{split}
\|y_{n+1}\|^2\leq\|y_{n}\|^2+(4h\sqrt{\lambda}+2\sqrt{\lambda}\|\Delta B_{n}\|^2+2\alpha_{n})(\|y_{n}\|^2+\|y_{[nh]m}\|^2)
\end{split}
\end{equation}
on $\{\omega\in\Omega:\|y_n\|\leq 1,~\|y_{[nh]m}\|\leq 1\}^c\cap\{\omega\in\Omega:\|y_n\|\leq m^{\frac{1}{2c}},~\|y_{[nh]m}\|\leq m^{\frac{1}{2c}}\}$ for all $n\in\{0,1,\cdots,Tm-1\}$ and $m\in\mathbb{N}$.

Before proving \eqref{dom}, we define the mapping
\begin{equation}\label{3.12}
  \tau_n(\omega)=\max\big\{\{-1\}\cup\{i\in\{0,1,2,\cdots,n\},\|y_i(\omega)\|\leq 1,\|y_{[ih]m}(\omega)\|\leq 1\}\big\}
\end{equation}
for all $\omega\in\Omega$, $n\in\{0,1,...,Tm\}$.
With the estimates \eqref{fore1} and \eqref{fore3} at hand, we now prove \eqref{dom} by induction. The base case $n=0$ is trivial. Now, let $n\in\{0,1,2,\cdots,Tm-1\}$ be fixed and arbitrary. Assume inequality \eqref{dom} hold for all $j\in\{0,1,2,\cdots,n\}$. Then we show that inequality \eqref{dom} holds for $j=n+1$, that is
\begin{equation}\label{n+1}
  \|y_{n+1}(\omega)\|\leq D_{n+1}(\omega)
\end{equation}
for all $\omega\in\Omega_{n+1}$. Let $\omega\in\Omega_{n+1}$ be arbitrary, because of the definition of $\Omega_{n}$, we can get $\omega\in\Omega_{n+1}\subset\Omega_j$ which indicates that $\|y_j(\omega)\|\leq D_j(\omega)\leq m^{\frac{1}{2c}}$ for all $j\in\{0,1,2,\cdots,n\}$, and it also follows from \eqref{3.12} that
\begin{equation*}
  1<\max\{\|y_i(\omega)\|,~\|y_{[ih]}(\omega)\|\}\leq m^{\frac{1}{2c}}
\end{equation*}
for $i=\tau_{n}(\omega)+1,\tau_{n}(\omega)+2,...,n$.

Case 1: if $\tau_n(\omega)\geq[nh]m$, then from \eqref{fore3}
\begin{equation*}
  \begin{split}
  &\|y_{n+1}(\omega)\|^2+\|y_{[nh]m}(\omega)\|^2\\
  \leq&\big(1+4h\sqrt{\lambda}+2\sqrt{\lambda}\|\Delta B_n\|^2+2\alpha_n\big)\big(\|y_{n}(\omega)\|^2+\|y_{[nh]m}(\omega)\|^2\big)\\
  \leq&(\|y_{n}(\omega)\|^2+\|y_{[nh]m}(\omega)\|^2)\exp\Big(4h\sqrt{\lambda}+2\sqrt{\lambda}\|\Delta B_{n}(\omega)\|^2+2\alpha_{n}(\omega)\Big)
  \\
  \leq&\ldots\\
  \leq&(\|y_{\tau_n+1}(\omega)\|^2+\|y_{[nh]m}(\omega)\|^2)\exp\Big(4h\sqrt{\lambda}(n-\tau_n)+2\sum_{i=\tau_n+1}^n(\sqrt{\lambda}\|\Delta B_i(\omega)\|^2+\alpha_i(\omega))\Big)
  \end{split}
\end{equation*}

Due to $\|y_{\tau_n}(\omega)\|\leq\lambda+\|\xi\|$, $\|y_{[nh]m}(\omega)\|\leq 1$, we get
\begin{equation}\label{3.14}
  \begin{split}
  \|y_{n+1}(\omega)\|^2+\|y_{[nh]m}(\omega)\|^2\leq(\lambda+\|\xi\|+1)^2\exp\Big(4h\sqrt{\lambda}(n-\tau_n)+2\sum_{i=\tau_n+1}^n(\sqrt{\lambda}\|\Delta B_i(\omega)\|^2+\alpha_i(\omega))\Big).
  \end{split}
\end{equation}
Therefore
\begin{equation}\label{fore4}
\begin{split}
 \|y_{n+1}(\omega)\|\leq(\lambda+\|\xi\|+1)\exp\Big(4T\sqrt{\lambda}+\sup_{u\in\{-1,0,1,\ldots,n\}}\sum_{i=u+1}^{n}(\sqrt{\lambda}\|\Delta B_{i}(\omega)\|^2+\alpha_{i}(\omega))\Big)\leq D_{n+1}(\omega).
 \end{split}
\end{equation}

Case 2: if $\tau_n(\omega)< [nh]m$, Using $2ab\leq a^2+b^2$, we derive $\|y_{n+1}\|$ from \eqref{fore3} and \eqref{3.14}, and obtain
\begin{equation}\label{3.15}
  \begin{split}
  &2\|y_{n+1}(\omega)\|\times\|y_{[nh]m}(\omega)\|\leq\|y_{n+1}(\omega)\|^2+\|y_{[nh]m}(\omega)\|^2\\
  \leq& 2\|y_{[nh]m}(\omega)\|^2\exp\Big(4h\sqrt{\lambda}(n-[nh]m)+2\sum_{i=[nh]m}^{n}\big(\sqrt{\lambda}\|\Delta B_i(\omega)\|^2+\alpha_i(\omega)\big)\Big).
  \end{split}
\end{equation}
Hence
\begin{equation}\label{fore5}
\begin{split}
  \|y_{n+1}(\omega)\|&\leq\|y_{[nh]m}(\omega)\|\exp\Big(4\sqrt{\lambda}h(n-[nh]m)+2\sum_{i=[nh]m}^{n}(\sqrt{\lambda}\|\Delta B_{i}(\omega)\|^2+\alpha_{i}(\omega))\Big)\\
  &\leq...\\
 &\leq\|y_{[\tau_nh+1]m}(\omega)\|\exp\Big(4h\sqrt{\lambda}(n-[\tau_nh+1]m)+2\sum_{i=[\tau_nh+1]m}^{n}(\sqrt{\lambda}\|\Delta B_i(\omega)\|^2+\alpha_i(\omega))\Big)\\
 &\leq(\|y_{\tau_n+1}(\omega)\|+\|y_{[\tau_nh]m}\|)\exp\Big(4h\sqrt{\lambda}(n-\tau_n)+2\sum_{i=[\tau_nh+1]m}^{n}(\sqrt{\lambda}\|\Delta B_i(\omega)\|^2+\alpha_i(\omega))\\
 &~~~~~~~~~~~~~~~~~~~~~~~~~~+\sum_{i=\tau_n+1}^{[\tau_nh+1]m}(\sqrt{\lambda}\|\Delta B_i(\omega)\|^2+\alpha_i(\omega))\Big)\\
 &\leq(\lambda+\|\xi\|+1)\exp\Big(4T\sqrt{\lambda}+2\sup_{u\in\{-1,0,1,\ldots,n\}}\sum_{i=u+1}^{n}(\sqrt{\lambda}\|\Delta B_{i}(\omega)\|^2+\alpha_{i}(\omega))\Big)\Big)\leq D_{n+1}(\omega).
  \end{split}
\end{equation}
The proof is completed.
\end{proof}
 The following two lemmas are useful to prove that $D_n$ is bounded on $\Omega$.

\begin{lemma}\label{lemma2}\cite{mh}
For all $p\geq 1$,
\begin{equation}\label{estimatBn}
  \sup_{m\geq 4\lambda p}\mathbb{E}\Big[\exp\big(p\lambda\sum_{i=0}^{Tm-1}\|\Delta B_i\|^2\big)\Big]<e^{2\lambda pTr}.
  \end{equation}
\end{lemma}
\begin{lemma}\label{lemma3}
Let $\alpha_{i}: \Omega\rightarrow \mathbb{R}$ for all $i\in\{0,1,\ldots,Tm\}$, and $m\in\mathbb{N}$ be given by \eqref{alphan}. Then for each $p\geq 1$,
\begin{equation}\label{fore8}
  \sup_{z\in\{-1,1\}}\sup_{m\in\mathbb{N}}\mathbb{E}\Big[\sup_{n\in\{0,1,\ldots,Tm\}}\exp\big(pz\sum_{i=0}^{n-1}\alpha_{i}\big)\Big]<2e^{2pT(K+\|\sigma(0,0)\|^2)}.
\end{equation}
\end{lemma}
\begin{proof}
Note that the time discrete stochastic process $z\sum_{i=0}^{n-1}\alpha_{i}$ is an $\mathcal{F}_{t_n}$-martingale for every $n\in\{0,1,\ldots,Tm\}$, $z\in\{-1,1\}$ and every $m\in\mathbb{N}$. Then it is easy to deduce that the time discrete stochastic process $\exp\big(z\sum_{i=0}^{n-1}\alpha_{i}\big)$ is a positive $\mathcal{F}_{t_n}$-submartingale for every $n\in\{0,1,\ldots,Tm\}$, $z\in\{-1,1\}$ and every $m\in\mathbb{N}$. Hence Doob's martingale inequality shows that
\begin{equation}\label{fore8}
  \Big\|\sup_{n\in\{0,1,\ldots,Tm\}}\exp\big(z\sum_{i=0}^{n-1}\alpha_{i}\big)\Big\|_{L_p(\Omega,\mathbb{R})}\leq\frac{p}{p-1}\Big\|\exp\big(z\sum_{i=0}^{Tm-1}\alpha_{i}\big)\Big\|_{L_p(\Omega,\mathbb{R})}
\end{equation}
for all $m\in\mathbb{N}$, $p\in(1,+\infty)$ and all $z\in\{-1,1\}$. Moreover, we have that
\begin{equation*}
\begin{split}
  &\mathbb{E}\Big[\big|pz1_{\{\|x\|\leq 1,\|y\|\leq 1\}^c}\big\langle\frac{x}{\sqrt{\|x\|^2+\|y\|^2}},\frac{\sigma(x,y)}{\sqrt{\|x\|^2+\|y\|^2}}\Delta B_i\big\rangle\big|^2\Big]\\
  \leq&p^2h1_{\{\|x\|^2\leq 1,\|y\|^2\leq 1\}^c}\frac{\|x^T\sigma(x,y)\|^2}{(\|x\|^2+\|y\|^2)^2}\\
  \leq&p^2h1_{\{\|x\|^2\leq 1,\|y\|^2\leq 1\}^c}\frac{2(K+\|\sigma(0,0)\|^2)\|x\|^2}{\|x\|^2+\|y\|^2}\\
  \leq&2p^2h(K+\|\sigma(0,0)\|)^2
  \end{split}
\end{equation*}
for all $x\in\mathbb{R}^d$, $i\in\{0,1,...,Tm-1\}$, $m\in\mathbb{N}$, $p\in[1,\infty)$ and all $z\in\{-1,1\}$. Therefore lemma 4.3 in \cite{hm} gives that
\begin{equation}\label{11}
  \begin{split}
  &\mathbb{E}\Big[\exp\big(pz1_{\{\|x\|\leq 1,\|y\|\leq 1\}^c}\big\langle\frac{x}{\sqrt{\|x\|^2+\|y\|^2}},\frac{\sigma(x,y)}{\sqrt{\|x\|^2+\|y\|^2}}\Delta B_i\big\rangle\big)\Big]\\
  \leq&\exp\big(2p^2h(K+\|\sigma(0,0)\|)^2\big)
  \end{split}
\end{equation}
for all $x\in\mathbb{R}^d$, $i\in\{0,1,...,Tm-1\}$, $m\in\mathbb{N}$, $p\in[1,\infty)$ and all $z\in\{-1,1\}$. In particular, (\ref{11}) shows that
$$\mathbb{E}\Big[\exp(pz\alpha_i)|\mathcal{F}_{t_i}\Big]\leq\exp\big(2p^2h(K+\|\sigma(0,0)\|)^2\big)$$
for all $i\in\{0,1,...,Tm-1\}$, $m\in\mathbb{N}$, $p\in[1,\infty)$ and all $z\in\{-1,1\}$. Hence, we obtain that
\begin{equation}\label{22}
  \mathbb{E}\Big[\exp\big(pz\sum_{i=0}^{Tm-1}\alpha_i\big)\Big]\leq\big(2p^2T(K+\|\sigma(0,0)\|)^2\big)
\end{equation}
for all $m\in\mathbb{N}$, $p\in[1,\infty)$ and all $z\in\{-1,1\}$. Combining (\ref{fore8}) and (\ref{22}), then for all $p\in[2,\infty)$
$$\sup_{z\in\{-1,1\}}\sup_{m\in\mathbb{N}}\Big\|\sup_{n\in\{0,1,...,Tm\}}\exp\big(z\sum_{i=0}^{n-1}\alpha_i\big)\Big\|_{L_p(\Omega,\mathbb{R})}\leq 2\exp\big(2pT(K+\|\sigma(0,0)\|)^2\big).$$
The proof is completed.
\end{proof}

\begin{lemma}\label{lemma4}
Let $D_{n}$ is defined by \eqref{dn}. Then for all $p\in[1,\infty)$
\begin{equation}\label{estimatedn}
  \sup_{m\in\mathbb{N}}\mathbb{E}\Big[\sup_{0\leq n\leq Tm} |D_{n}|^p\Big]<\infty.
\end{equation}
\end{lemma}
\begin{proof}
The proof is similar to that of lemma 3.5 in \cite{mh}.
\end{proof}
\begin{lemma}\label{lemma5}
Let $\Omega_{Tm}\in\mathcal{F}$ for $m\in\mathbb{N}$ be given by \eqref{wn}. Then for each $p\geq 1$ we have
\begin{equation}\label{wc}
  \sup_{m\in\mathbb{N}}\big(m^p\cdot\mathbb{P}[(\Omega_{Tm})^c]\big)<\infty.
\end{equation}

\end{lemma}
\begin{proof}
This Lemma is based on the Lemma \ref{lemma3} and \ref{lemma4} and the proof is same as that of Lemma 3.6 in \cite{mh}.
\end{proof}
Next, we will prove the boundedness of $y_n$ in $L_p$ sense.
\begin{theorem}\label{th1}
Let $y_{n}: \Omega\rightarrow\mathbb{R}^d$, for $n\in\{0,1,\cdots,Tm\}$ and $m\in\mathbb{N}$ be given by \eqref{yn}, then for all $p\in[1,\infty)$
\begin{equation}\label{bound}
  \sup_{m\in\mathbb{N}}\Big[\sup_{0\leq n\leq Tm}\mathbb{E}\|y_{n}\|^p\Big]<\infty.
\end{equation}
\end{theorem}
\begin{proof}
First, we can by \eqref{yn} represent the approximation $y_{n}$ as following
\begin{equation*}
\begin{split}
  y_{n}=&\xi+\sigma(0,0)B_{n}+\sum_{i=0}^{n-1}\frac{h\mu(y_{i},y_{[ih]m})}{1+h\|\mu(y_{i},y_{[ih]m})\|}+\sum_{i=0}^{n-1}\big(\sigma(y_{i},y_{[ih]m})-\sigma(0,0)\big)\Delta B_{i}
 \end{split}
\end{equation*}
for all $n\in\{0,1,\ldots,Tm\},$ and $m\in\mathbb{N}$. The Lemma 4.7 in \cite{hm} and Burkholder-Davis-Gundy type inequality in Lemma 3.8 in \cite{mh} then give that
\begin{equation*}
  \begin{split}
  \|y_{n}\|_{L_{p}(\Omega,\mathbb{R})}\leq&\|\xi\|_{L_{p}(\Omega,\mathbb{R})}+p\sqrt{rT}\|\sigma(0,0)\|+Tm
  \\
  &+p\Big(\sum_{i=0}^{n-1}\sum_{\nu=1}^{r}\|\sigma_{\nu}(y_{i},y_{[ih]m})-\sigma_\nu(0,0)\|^2_{L_{p}(\Omega,\mathbb{R})}h\Big)^\frac{1}{2}
 \\
 \leq&\big(\|\xi\|_{L_{p}(\Omega,\mathbb{R})}+p\sqrt{rT}\|\sigma(0,0)\|+Tm\big)\\
 &+pK\sqrt{2rh}\Big(\sum_{i=0}^{n-1}\big(\|y_{i}\|^2_{L_{p}(\Omega,\mathbb{R})}+\|y_{[ih]m}\|^2_{L_{p}(\Omega,\mathbb{R})}\big)\Big)^\frac{1}{2}.
  \end{split}
\end{equation*}
Using Gronwall's inequality, we can get
\begin{equation*}
  \begin{split}
  \|y_{n}\|^2_{L_{p}(\Omega,\mathbb{R})}\leq&2\big(\|\xi\|_{L_{p}(\Omega,\mathbb{R})}+p\sqrt{rT}\|\sigma(0,0)\|+Tm\big)^2
  \\&+8p^2K^2rh\sum_{i=0}^{n-1}\sup_{j\in\{0,1,...,i\}}\|y_{j}\|^2_{L_{p}(\Omega,\mathbb{R})}\\
  \leq&2\big(\|\xi\|_{L_{p}(\Omega,\mathbb{R})}+p\sqrt{rT}\|\sigma(0,0)\|+Tm\big)^2e^{8p^2K^2Tr}
  \end{split}
\end{equation*}
for all $n\in\{0,1,\ldots,Tm\}, m\in\mathbb{N}$ and $p\in[2,\infty)$. For all $n\in\{0,1,2,...,Tm\}$, $m\in\mathbb{N}$ and $p\in[2,\infty)$, it is easy to obtain that

\begin{equation}\label{fore16}
 \|y_{n}\|_{L_{p}(\Omega,\mathbb{R})}\leq \sqrt{2}\big(\|\xi\|_{L_{p}(\Omega,\mathbb{R})}+p\sqrt{rT}\|\sigma(0,0)\|+Tm\big)e^{4p^2K^2Tr}.
\end{equation}
Of course \eqref{fore16} doesn't prove $\|y_{n}\|_{L_{p}(\Omega,\mathbb{R})}<\infty$, due to $m\in\mathbb{N}$ on the right-hand side of \eqref{fore16}. However, H\"{o}lder inequality and lemma 3.5 show that
\begin{equation}\label{3.25}
  \begin{split}
  &\sup_{m\in\mathbb{N}}\sup_{n\in\{0,1,\ldots,Tm\}}\|1_{(\Omega_{n})^c}y_{n}\|_{L_{p}(\Omega,\mathbb{R})}
  \\
  \leq&\sup_{m\in\mathbb{N}}\sup_{n\in\{0,1,\ldots,Tm\}}(m\|1_{(\Omega_{n})^c}\|_{L_{2p}(\Omega,\mathbb{R})})\times\Big(\sup_{m\in\mathbb{N}}\sup_{n\in\{0,1,\ldots,Tm\}}m^{-1}\|y_{n}\|_{L_{2p}(\Omega,\mathbb{R})}\Big)
  \\
  \leq&\sqrt{2}e^{4p^2K^2Tr}\Big(\sup_{m\in\mathbb{N}}m^{2p}\cdot\mathbb{P}[(\Omega_{Tm})^c]\Big)^\frac{1}{2p}\times\big(\|\xi\|_{L_{2p}(\Omega,\mathbb{R})}+p\sqrt{rT}\|\sigma(0,0)\|+T\big)
  \\
 <&\infty
  \end{split}
\end{equation}
for all $p\in[2,\infty)$. Additionally, Lemmas \ref{lemma1} and \ref{lemma4} give that
\begin{equation}\label{3.26}
 \sup_{m\in\mathbb{N}}\sup_{n\in\{0,1,\ldots,Tm\}}\|1_{\Omega_{n}}y_{n}\|_{L_{p}(\Omega,\mathbb{R})}\leq\sup_{m\in\mathbb{N}}\sup_{n\in\{0,1,\ldots,Tm\}}\|D_{n}\|_{L_{p}(\Omega,\mathbb{R})}<\infty
\end{equation}
for all $p\in[2,\infty)$. Combining (\ref{3.25}) and (\ref{3.26}) the theorem is proved.
\end{proof}
\begin{lemma}\label{6}
Let $y_{n}: \Omega\rightarrow\mathbb{R}^d$, for $n\in\mathbb{N}=\{1,2,...,Tm\}$ and $m\in\mathbb{N}$ be given by \eqref{yn}, then for all $p\in[1,\infty)$
\begin{equation}\label{fore17}
\begin{split}
  \sup_{m\in\mathbb{N}}\sup_{n\in\{0,1,\ldots,Tm\}}\mathbb{E}[\|\mu(y_{n},y_{[nh]m})\|^p]<\infty,~\sup_{m\in\mathbb{N}}\sup_{n\in\{0,1,\ldots,Tm\}}\mathbb{E}[\|\sigma(y_{n},y_{[nh]m})\|^p]<\infty.
\end{split}
\end{equation}
\end{lemma}
\begin{proof}
By condition \eqref{con4}, we can get for any $x,y\in\mathbb{R}^d$
\begin{equation*}\begin{split}
 \|\mu(x,y)\|\leq&\|\mu(0,0)\|+\|\mu_x(\theta_1x,y)\|\|x\|+K\|y\|\\
 \leq&\|\mu(0,0)\|+K(1+\|x\|^c)\|x\|+K\|y\|\\
 \leq&\|\mu(0,0)\|+2K(1+\|x\|^{c+1})+K\|y\|.
 \end{split}
\end{equation*}
It comes from theorem \ref{th1} that
\begin{equation*}
  \begin{split}
  \sup_{m\in\mathbb{N}}&\sup_{n\in\{0,1,\ldots,Tm\}}\|\mu(y_{n},y_{[nh]m})\|_{L_{p}(\Omega,\mathbb{R})}\leq\|\mu(0,0)\|+2K\Big(1+ \sup_{m\in\mathbb{N}}\sup_{n\in\{0,1,\ldots,Tm\}}\|y_{n}\|^{c+1}_{L_{(c+1)p}(\Omega,\mathbb{R})}\Big)\\
  &+K \sup_{m\in\mathbb{N}}\sup_{n\in\{0,1,\ldots,Tm\}}\|y_{[nh]m}\|_{L_{p}(\Omega,\mathbb{R})}<\infty
  \end{split}
\end{equation*}
for all $p\in[1,\infty)$. According to the global-Lipschiz condition (\ref{con1}) and theorem \ref{th1}, for any $x,y\in\mathbb{R}^d$, we obtain\\
\begin{equation*}
  \begin{split}
  \sup_{m\in\mathbb{N}}&\sup_{n\in\{0,1,\ldots,Tm\}}\|\sigma(y_{n},y_{[nh]m})\|_{L_{p}(\Omega,\mathbb{R})}\leq\|\sigma(0,0)\|+K\Big(\sup_{m\in\mathbb{N}}\sup_{n\in\{0,1,\ldots,Tm\}}\|y_{n}\|_{L_{p}(\Omega,\mathbb{R})}\Big)\\
  &+K\Big( \sup_{m\in\mathbb{N}}\sup_{n\in\{0,1,\ldots,Tm\}}\|y_{[nh]m}\|_{L_{p}(\Omega,\mathbb{R})}\Big)<\infty.
  \end{split}
\end{equation*}
\end{proof}

Now we are in a position to give the proof of theorem \ref{th2}
\begin{proof}
We define $\underline{t}=t_{n}$ for any $t\in[t_{n},t_{n+1})$, and $n\in\{0,1,\ldots,Tm-1\}$, $m\in\mathbb{N}$. It is known from (\ref{yt})
\begin{equation}\label{yt0}
  y(t)=\xi+\int_{0}^{t}\frac{\mu(y(\underline{s}),y([\underline{s}]))}{1+h\|\mu(y(\underline{s}),y([\underline{s}]))\|}ds+\int_{0}^{t}\sigma(y(\underline{s}),y([\underline{s}]))dB(s)
\end{equation}
for all $t\in[0,T]$. Note that
\begin{equation*}
\begin{split}
  x(t)-y(t)=&\int_{0}^{t}\Big(\mu(x(s),x([s]))-\frac{\mu(y(\underline{s}),y([\underline{s}]))}{1+h\|\mu(y(\underline{s}),y([\underline{s}]))\|}\Big)ds\\
  &+\sum_{\nu=1}^{r}\int_{0}^{t}\Big(\sigma_\nu(x(s),x([s])-\sigma_\nu(y(\underline{s}),y([\underline{s}]))\Big)dB_\nu(s)
\end{split}
\end{equation*}
for all $t\in[0,T]$ $\mathbb{P}$-a.s.. Hence, It\^{o}'s formula yields that
\begin{equation*}
  \begin{split}
  \|x(t)-y(t)\|^2=&2\int_{0}^{t}\big<x(s)-y(s),\mu(x(s),x([s]))-\mu(y(\underline{s}),y([\underline{s}]))\big>ds\\
  &+2h\int_{0}^{t}\Big<x(s)-y(s),\frac{\mu(y(\underline{s}),y([\underline{s}]))\|\mu(y(\underline{s}),y([\underline{s}]))\|}{1+h\|\mu(y(\underline{s}),y([\underline{s}]))\|}\Big>ds\\
  &+\sum_{\nu=1}^{r}\int_{0}^{t}\big\|\sigma_\nu(x(s),x([s]))-\sigma_\nu(y(\underline{s}),y([\underline{s}]))\big\|^2ds\\
  &+2\sum_{\nu=1}^{r}\int_{0}^{t}\big<x(s)-y(s),\sigma_\nu(x(s),x([s]))-\sigma_\nu(y(\underline{s}),y([\underline{s}]))\big>dB_\nu(s)\\
  =&A_1+A_2+A_3+A_4
    \end{split}
\end{equation*}
for $A_1$, $A_2$, $A_3$, we can estimate
\begin{equation}\label{main1}
  \begin{split}
  A_1=&2\int_{0}^{t}\big<x(s)-y(s),\mu(x(s),x([s]))-\mu(y(s),x([s]))\big>ds\\
  &+2\int_{0}^{t}\big<x(s)-y(s),\mu(y(s),x([s]))-\mu(y(s),y([\underline{s}]))\big>ds\\
  &+2\int_{0}^{t}\big<x(s)-y(s),\mu(y(s),y([\underline{s}]))-\mu(y(\underline{s}),y([\underline{s}]))\big>ds\\
  \leq&(2K+2)\int_{0}^{t}\|x(s)-y(s)\|^2ds+K^2\int_{0}^{t}\big\|x([s])-y([\underline{s}])\big\|^2ds\\
  &+\int_{0}^{t}\big\|\mu(y(s),y([\underline{s}]))-\mu(y(\underline{s}),y([\underline{s}]))\big\|^2ds,
  \end{split}
\end{equation}
\begin{equation}\label{main2}
\begin{split}
 A_2\leq&\int_{0}^{t}\|x(s)-y(s)\|^2ds+h^2\int_{0}^{t}\|\mu(y(\underline{s}),y([\underline{s}]))\|^4ds,~~~~~~~~~~~~
\end{split}
\end{equation}
\begin{equation}\label{main3}
  \begin{split}
 A_3\leq&2\sum_{\nu=1}^{r}\int_{0}^{t}\big\|\sigma_\nu(x(s),x([s])-\sigma_\nu(y(s),y([\underline{s}]))\big\|^2ds\\
  &+2\sum_{\nu=1}^{r}\int_{0}^{t}\big\|\sigma_\nu(y(s),y([\underline{s}])))-\sigma_\nu(y(\underline{s}),y([\underline{s}]))\big\|^2ds\\
  \leq&4K^2r\int_{0}^{t}\big\|x(s)-y(s)\big\|^2ds+4K^2r\int_{0}^{t}\big\|x([s])-y([\underline{s}])\big\|^2ds~~
  \end{split}
\end{equation}
$$  +2K^2r\int_{0}^{t}\big\|y(s)-y(\underline{s})\big\|^2ds.$$
The Global-inequality, \eqref{main1}, \eqref{main2}, \eqref{main3} and $2ab\leq a^2+b^2$ shows that
\begin{equation}\label{mainsup}
  \begin{split}
  \sup_{t\in[0,t_1]}&\big\|x(t)-y(t)\big\|^2\leq(4K^2r+2K+3)\int_{0}^{t_1}\|x(s)-y(s)\|^2ds\\
  &+(4r+1)K^2\int_{0}^{t_1}\|x([s])-y([\underline{s}])\|^2ds+h^2\int_{0}^{T}\|\mu(y(\underline{s}),y([\underline{s}]))\|^4ds\\
  &+\int_{0}^{T}\big\|\mu(y(s),y([\underline{s}]))-\mu(y(\underline{s}),y([\underline{s}]))\big\|^2ds+2K^2r\int_{0}^{T}\big\|y(s)-y(\underline{s})\big\|^2ds\\
  &+2\sup_{t\in[0,t_1]}\Big|\sum_{\nu=1}^{r}\int_{0}^{t}\big<x(s)-y(s),\sigma_\nu(x(s),x([s]))-\sigma_\nu(y(\underline{s}),y([\underline{s}]))\big>dB_\nu(s)\Big|
  \end{split}
\end{equation}
$\mathbb{P}$-a.s. for all $t_1\in[0,T]$. The Burkholder-Davis-Gundy type inequality in Lemma 3.7 in \cite{mh} hence yields that
\begin{equation}\label{3.39}
  \begin{split}
  \Big\|\sup_{t\in[0,t_1]}&\big\|x(t)-y(t)\big\|^2\Big\|_{L_{\frac{p}{2}}(\Omega,\mathbb{R})}\leq\int_{0}^{T}\big\|\mu(y(s),y([\underline{s}]))-\mu(y(\underline{s}),y([\underline{s}]))\big\|^2_{L_{p}(\Omega,\mathbb{R})}ds\\
 &+(8K^2r+K^2+2K+3)\int_{0}^{t_1}\Big(\sup_{u\in[0,s]}\big\|x(u)-y(u)\big\|^2_{L_{p}(\Omega,\mathbb{R})}\Big)ds\\
  &+2K^2r\int_{0}^{T}\big\|y(s)-y(\underline{s})\big\|^2_{L_{p}(\Omega,\mathbb{R})}ds+h^2\int_{0}^{T}\|\mu(y(\underline{s}),y([\underline{s}]))\|^4_{L_{2p}(\Omega,\mathbb{R})}ds\\
  &+\underbrace{2p\Big(\sum_{\nu=1}^{r}\int_{0}^{t_1}\big\|\big<x(s)-y(s),\sigma_\nu(x(s),x([s])-\sigma_\nu(y(\underline{s}),y([\underline{s}]))\big>\big\|^2_{L_\frac{p}{2}(\Omega,\mathbb{R})}ds\Big)^\frac{1}{2}}_\alpha
  \end{split}
\end{equation}
for all $t_1\in[0,T],$ and all $p\in[4,\infty)$. Next the Cauchy-Schwarz inequality, H\"{o}lder inequality and again the inequality $2ab\leq a^2+b^2$ for all $a,b\in\mathbb{R}$ imply that
\begin{equation}
  \begin{split}
\alpha\leq&2p\Big(\sum_{\nu=1}^{r}\int_{0}^{t_1}\big\|x(s)-y(s)\big\|^2_{L_p(\Omega,\mathbb{R})}\big\|\sigma_\nu(x(s),x([s]))-\sigma_\nu(y(\underline{s}),y([\underline{s}]))\big\|^2_{L_p(\Omega,\mathbb{R})}ds\Big)^\frac{1}{2}\\
  \leq&2p\Big(\sup_{s\in[0,t_1]}\big\|x(s)-y(s)\big\|_{L_p(\Omega,\mathbb{R})}\Big)\Big(\sum_{\nu=1}^{r}\int_{0}^{t_1}\big\|\sigma_\nu(x(s),x([s]))-\sigma_\nu(y(\underline{s}),y([\underline{s}]))\big\|^2_{L_p(\Omega,\mathbb{R})}ds\Big)^\frac{1}{2}\\
  \leq&\frac{1}{2}\sup_{t\in[0,t_1]}\big\|x(t)-y(t)\big\|^2_{L_p(\Omega,\mathbb{R})}+4p^2K^2r\int_{0}^{t_1}\big(\big\|x(s)-y(\underline{s})\big\|^2_{L_p(\Omega,\mathbb{R})}+\big\|x([s])-y([\underline{s}])\big\|^2_{L_p(\Omega,\mathbb{R})}\big)ds\\ \leq&\frac{1}{2}\sup_{t\in[0,t_1]}\big\|x(t)-y(t)\big\|^2_{L_p(\Omega,\mathbb{R})}+8p^2K^2r\int_{0}^{t_1}\big\|x(s)-y(s)\big\|^2_{L_p(\Omega,\mathbb{R})}ds\\
  &+8p^2K^2r\int_{0}^{t_1}\big\|y(s)-y(\underline{s})\big\|^2_{L_p(\Omega,\mathbb{R})}ds+4p^2K^2r\int_{0}^{t_1}\big\|x([s])-y([s])\big\|^2_{L_p(\Omega,\mathbb{R})}ds
  \end{split}
\end{equation}
for all $t\in[0,T],$ and all $p\in[4,\infty)$. Inserting inequality above into \eqref{3.39} and applying the estimate $(a+b)^2\leq 2a^2+2b^2$ for all $a,b\in\mathbb{R}$, then yield that
\begin{equation}\label{3.40}
  \begin{split}
\frac{1}{2}\Big\|&\sup_{t\in[0,t_1]}\big\|x(t)-y(t)\big\|\Big\|^2_{L_{p}(\Omega,\mathbb{R})}\leq(2K^2r+8p^2K^2r)\int_{0}^{T}\big\|y(s)-y(\underline{s})\big\|^2_{L_{p}(\Omega,\mathbb{R})}ds\\
&+(8K^2r+K^2+2K+3+12p^2K^2r)\int_{0}^{t_1}\Big(\sup_{u\in[0,s]}\big\|x(u)-y(u)\big\|^2_{L_{p}(\Omega,\mathbb{R})}\Big)ds\\
   &+h^2\int_{0}^{T}\|\mu(y(\underline{s}),y([\underline{s}]))\|^4_{L_{2p}(\Omega,\mathbb{R})}ds+\int_{0}^{T}\big\|\mu(y(s),y([\underline{s}))-\mu(y(\underline{s}),y([\underline{s}]))\big\|^2_{L_{p}(\Omega,\mathbb{R})}ds\\
  \end{split}
\end{equation}
for all $t\in[0,T],$ and all $p\in[4,\infty)$. In the next step Gronwall's Lemma shows that
\begin{equation}\label{3.41}
  \begin{split}
  \Big\|\sup_{t\in[0,t_1]}\big\|x(t)-y(t)&\big\|\Big\|^2_{L_{p}(\Omega,\mathbb{R})}\leq 2e^{T(8K^2r+K^2+2K+3+12p^2K^2r)}\times \Big((2K^2r+8p^2K^2r)\\
  &\times\int_{0}^{T}\big\|y(s)-y(\underline{s})\big\|^2_{L_{p}(\Omega,\mathbb{R})}ds+h^2\int_{0}^{T}\big\|\mu(y(\underline{s}),y([\underline{s}]))\big\|^4_{L_{2p}(\Omega,\mathbb{R})}ds\\
  &+\int_{0}^{T}\big\|\mu(y(s),y([\underline{s}]))-\mu(y(\underline{s}),y([\underline{s}]))\big\|^2_{L_{p}(\Omega,\mathbb{R})}ds\Big)
  \end{split}
\end{equation}
and hence, the inequality $\sqrt{a+b+c}\leq\sqrt{a}+\sqrt{b}+\sqrt{c}$ for all $a,b,c\in[0,\infty)$ gives that
\begin{equation}\label{*}
  \begin{split}
  &\Big\|\sup_{t\in[0,t_1]}\big\|x(t)-y(t)\big\|\Big\|_{L_{p}(\Omega,\mathbb{R})}\\
  \leq& \sqrt{2T}e^{\frac{1}{2}T(8K^2r+K^2+2K+3+12p^2K^2r)}\times\Big(\sqrt{2K^2r+8p^2K^2r}\Big[\sup_{t\in[0,T]}\big\|y(t)-y(\underline{t})\big\|_{L_{p}(\Omega,\mathbb{R})}\Big]\\
  &+h\Big[\sup_{t\in[0,T]}\big\|\mu(y(\underline{t}),y([\underline{t}]))\big\|^2_{L_{2p}(\Omega,\mathbb{R})}\Big]+\sup_{t\in[0,T]}\big\|\mu(y(t),y([\underline{t}]))-\mu(y(\underline{t}),y([\underline{t}]))\big\|_{L_{p}(\Omega,\mathbb{R})}\Big)
  \end{split}
\end{equation}
for all $N\in\mathbb{N}$ and all $p\in[4,\infty)$. Additionally, the Burkholder-Davis-Gundy inequality in Lemma 3.7 in \cite{mh} shows that
\begin{equation*}\begin{split}
  \sup_{t\in[0,T]}&\big\|y(t)-y(\underline{t})\big\|_{L_{p}(\Omega,\mathbb{R})}\leq h\Big(\sup_{t\in[0,T]}\big\|\mu(y(\underline{t}),y([\underline{t}]))\big\|\Big)_{L_{p}(\Omega,\mathbb{R})}\\
  &+\sup_{t\in[0,T]}\big\|\int_{\underline{t}}^{t}\sigma(y(\underline{s}),y([\underline{s}]))dB(s)\big\|_{L_{p}(\Omega,\mathbb{R})}\\
  \leq&h\Big(\sup_{n\in\{0,1,\ldots,Tm\}}\big\|\mu(y_{n},y_{[nh]m}\big\|\Big)_{L_{p}(\Omega,\mathbb{R})}\\
  &+p\Big(\int_{\underline{t}}^{t}\sum_{\nu=1}^{r}\big\|\sigma_\nu(y(\underline{s}),y([\underline{s}]))\big\|^2_{L_{p}(\Omega,\mathbb{R})}ds\Big)^\frac{1}{2}\\
  \leq&h\Big(\sup_{n\in\{0,1,\ldots,Tm\}}\big\|\mu(y_{n},y_{[nh]m}\big\|\Big)_{L_{p}(\Omega,\mathbb{R})}\\
  &+p\sqrt{hr}\Big(\sup_{\nu\in\{0,1,\ldots,r\}}\sup_{n\in\{0,1,\ldots,Tm\}}\big\|\sigma_\nu(y_{n},y_{[nh]m}\big\|\Big)_{L_{p}(\Omega,\mathbb{R})}
  \end{split}
\end{equation*}
for all $m\in\mathbb{N}$ and all $p\in[2,\infty)$. Lemma \ref{6} implies that, for all $p\in[1,\infty)$, there exists $C_1>0$ independent of $h$ such that
\begin{equation}\label{ny}
\sup\limits_{t\in[0,T]}\big\|y(t)-y(\underline{t})\big\|_{L_{p}(\Omega,\mathbb{R})}<C_1h^{\frac{1}{2}}.
\end{equation}
In particular, we obtain that for all $p\in[1,\infty)$
\begin{equation}\label{ty}
\sup\limits_{t\in[0,T]}\big\|y(t)\big\|_{L_{p}(\Omega,\mathbb{R})}<\infty.
\end{equation}
Moreover, the estimate
 \begin{equation}\label{supmu}
   \begin{split}
   &\sup_{t\in[0,T]}\big\|\mu(y(t),y([\underline{t}]))-\mu(y(\underline{t}),y([\underline{t}]))\big\|_{L_{p}(\Omega,\mathbb{R})}\\
   \leq&K\Big(1+2\sup_{t\in[0,T]}\big\|y(t)\big\|^c_{L_{2cp}(\Omega,\mathbb{R})}\Big)\Big(\sup_{t\in[0,T]}\big\|y(t)-y(\underline{t})\big\|_{L_{2p}(\Omega,\mathbb{R})}\Big)
   \end{split}
 \end{equation}
 for all $p\in[1,\infty)$. Then inequality \eqref{ny} and  \eqref{ty} hence show that there exists $C_2>0$ independent of $h$
 \begin{equation}\label{3.43}
\sup_{t\in[0,T]}\big\|\mu(y(t),y([\underline{t}]))-\mu(y(\underline{t}),y([\underline{t}]))\big\|_{L_{p}(\Omega,\mathbb{R})}\leq C_2h^{\frac{1}{2}}
 \end{equation}
 for all $p\in[1,\infty)$. We obtain from \eqref{*}, \eqref{ny}, \eqref{3.43} and lemma \ref{6} that there exists non-negative constant $C$ independent of $h$ such that \eqref{conclusion} holds. The proof is complete.
\end{proof}
\section{Numerical Experiments}
\setcounter{equation}{0}
\noindent
In this section, we give two numerical experiments to illustrate the strong convergence and the convergence order.

we consider
\begin{eqnarray}\label{4.1}
\begin{cases}
dx(t)=\big(-x(t)^\alpha+a(x(t)+x([t]))\big)dt+b(x(t)+x([t]))dB(t),~t\in[0,T],\\
x(0)=c.
\end{cases}
\end{eqnarray}

In the first numerical experiment we used the parameters~$\alpha=3$, $a=0.5$, $b=1$, $c=1.5$. In the second numerical experiment, we used parameters $\alpha=5$, $a=4.5$, $b=3$, $c=1$.

We square both sides of \eqref{conclusion} with $p=2$, we get the mean square error $\mathbb{E}\big[\sup_{t\in[0,T]}\|x(t)-y(t)\|^2\big]$ which should be bounded by $Ch$. The mean square error at time $T$ was estimated in the following way. A set of $30$ blocks each containing $100$ outcomes($\omega_{ij}:~1\leq i\leq 30,~1\leq j\leq 100$) were simulated. We denoted by $y(T,\omega_{ij})$ the numerical solution of the $j$th trajectory in the $i$th blocks and $x(T,\omega_{ij})$ the exact solution of \eqref{4.1} in the $j$th trajectory and $i$th block. The 'exact solution' was computed on a very fine mesh(we used $262144$ step).

Let $\epsilon$ denote the mean square error. Then by the law of large numbers, we conclude that
\begin{equation*}
\epsilon(T)=\frac{1}{3000}\sum_{i=1}^{30}\sum_{j=1}^{100}\|x(T,\omega_{ij})-y(T,\omega_{ij})\|^2.
\end{equation*}
There are three test in each numerical experiment with $T=1,2,3$. We can see from the table 1 and table 2, the ratios of errors in the tables are consistent with the theoretical rate of convergence as stated in theorem \ref{th2}.
\begin{table}
  \centering
\scriptsize
 \begin{tabular}{|c|c|c|c|c|c|c|c|c|c|c|}  
\hline
step~~~&$\epsilon(1)$~~~&ratio~~~&$\epsilon(2)$~~~&ratio~~~&$\epsilon(3)$~~~&ratio\\
\hline
$2^{-8}$~~~&$0.0022$~~~&$*$~~~&$0.0038$~~~&$*$~~~&$0.0089$~~~&$*$\\
\hline
$2^{-9}$~~~&$0.0010$~~~&$2.2000$~~~&$0.0020$~~~&$1.9000$~~~&$0.0051$~~~&$1.7451$\\
\hline
$2^{-10}$~~~&$0.0005$~~~&$2.0000$~~~&$0.0010$~~~&$2.0000$~~~&$0.0018$~~~&$2.8883$\\
\hline
$2^{-11}$~~~&$0.0002$~~~&$2.5000$~~~&$0.0006$~~~&$1.6667$~~~&$0.0009$~~~&$2.0000$\\
\hline
$2^{-12}$~~~&$0.0001$~~~&$2.0000$~~~&$0.0003$~~~&$2.0000$~~~&$0.0004$~~~&$2.2500$\\
\hline
\end{tabular}
\caption{The error at times $T=1,2,3$, for the first numerical experiment.}\label{1}
\end{table}
\begin{table}
  \centering
\scriptsize
 \begin{tabular}{|c|c|c|c|c|c|c|c|c|c|c|}  
\hline
step~~~&$\epsilon(1)$~~~&ratio~~~&$\epsilon(2)$~~~&ratio~~~&$\epsilon(3)$~~~&ratio\\
\hline
$2^{-8}$~~~&$0.0379$~~~&$*$~~~&$0.1550$~~~&$*$~~~&$0.2779$~~~&$*$\\
\hline
$2^{-9}$~~~&$0.0150$~~~&$2.5252$~~~&$0.0844$~~~&$1.8359$~~~&$0.1311$~~~&$2.1198$\\
\hline
$2^{-10}$~~~&$0.0073$~~~&$2.0592$~~~&$0.0443$~~~&$1.9068$~~~&$0.0808$~~~&$1.6225$\\
\hline
$2^{-11}$~~~&$0.0033$~~~&$2.2375$~~~&$0.0158$~~~&$2.8107$~~~&$0.0471$~~~&$1.7155$\\
\hline
$2^{-12}$~~~&$0.0015$~~~&$2.1252$~~~&$0.0107$~~~&$1.4727$~~~&$0.0401$~~~&$1.1746$\\
\hline
\end{tabular}
\caption{The error at times $T=1,2,3$, for the second numerical experiment.}\label{2}
\vspace{-3mm}
\end{table}

\end{document}